\documentclass{amsart} 

\newtheorem{theorem}{Theorem}[section]
\newtheorem{lemma}[theorem]{Lemma}
\newtheorem{proposition}[theorem]{Proposition}
\newtheorem{corollary}[theorem]{Corollary}

\theoremstyle{definition}
\newtheorem{definition}[theorem]{Definition}
\newtheorem{example}[theorem]{Example}
\newtheorem{remark}[theorem]{Remark}

\allowdisplaybreaks

\newcommand{\Z}{\mathbb{Z}}
\newcommand{\Q}{\mathbb{Q}}

\newcommand{\e}{\varepsilon}

\makeatletter

\@addtoreset{figure}{section}
\makeatother

\makeatletter

\@addtoreset{table}{section}
\makeatother

\makeatletter
  
  \@addtoreset{equation}{section}
\makeatother

\usepackage{longtable,caption} 
\usepackage[dvipdfmx]{graphicx}
\usepackage{comment} 
\usepackage{lscape}
\usepackage{color}
\usepackage{amsmath,amssymb}
\usepackage{physics}
\usepackage[abs]{overpic}
\usepackage{multirow}

\begin{document}


\title[The $V_1$- and $V_2$-polynomials of a long virtual knot]
{The $V_1$- and $V_2$-polynomials of a long virtual knot}

\author[Shin SATOH]{Shin SATOH}
\address{Department of Mathematics, Kobe University, 
Rokkodai-cho 1-1, Nada-ku, Kobe 657-8501, Japan}
\email{shin@math.kobe-u.ac.jp}

\author[Kodai WADA]{Kodai WADA}
\address{Department of Mathematics, Kobe University, 
Rokkodai-cho 1-1, Nada-ku, Kobe 657-8501, Japan}
\email{wada@math.kobe-u.ac.jp}

\makeatletter
\@namedef{subjclassname@2020}{%
  \textup{2020} Mathematics Subject Classification}
\makeatother


\thanks{This work was supported by JSPS KAKENHI Grant Numbers JP22K03287 and JP23K12973.}


\begin{abstract} 
We introduce two polynomial invariants 
$V_1(K;t)$ and $V_2(K;t)$ of a long virtual knot $K$, 
which generalize the degree-two finite type invariants $v_{2,1}$ and $v_{2,2}$ 
of Goussarov, Polyak, and Viro. 
We establish their fundamental properties 
and show that any pair of Laurent polynomials
can be realized as $(V_1(K;t),V_2(K;t))$ for some long virtual knot $K$.  
While these polynomials are not finite type invariants of any degree 
with respect to virtualizations, 
their first derivatives at $t=1$ define finite type invariants of degree three. 
As an application, 
we obtain an explicit Gauss diagram formula for the $\alpha_3$-invariant. 
\end{abstract} 

\maketitle


\section{Introduction}\label{sec1} 

Finite type invariants play a central role in knot theory,
providing a powerful framework for understanding knot invariants via a filtration by degree. 
Since their introduction by Vassiliev~\cite{Vas}, 
through the study of the zeroth cohomology of the space of 
smooth embeddings of $S^1$ into $\mathbb{R}^3$,
finite type invariants have influenced a wide range of developments in low-dimensional topology; 
see, for example,~\cite{Bar,CDM,Gou,Hab,Oht}. 
An equivalent and more combinatorial formulation
in terms of crossing changes was later given by Birman and Lin~\cite{BL},
and Kauffman~\cite{Kau} extended this perspective to the setting of virtual knots.

In~\cite{GPV}, Goussarov, Polyak, and Viro 
defined a new notion of finite type invariants 
for virtual knots and long virtual knots
using the operation of virtualizations.
They showed that any such invariant of degree $m$
is a finite type invariant of degree at most $m$ 
with respect to crossing changes,
and conversely that any finite type invariant of classical knots 
with respect to crossing changes 
arises as a specialization of a finite type invariant 
of long virtual knots with respect to virtualizations~\cite{GPV}.
This establishes long virtual knots
as a natural and important framework
for understanding finite type phenomena
in both classical and virtual knot theories,  
and motivates the systematic study of invariants of long virtual knots.

In degree two, the space of finite type invariants 
with respect to virtualizations 
for long virtual knots is two-dimensional~\cite{GPV}.
More concretely, Goussarov, Polyak, and Viro constructed
two integer-valued finite type invariants,
denoted by $v_{2,1}(K)$ and $v_{2,2}(K)$,
defined by Gauss diagram formulae. 
These invariants coincide with the second coefficient of the Conway polynomial
for long classical knots, but are independent in the virtual category.

The main purpose of this paper is to introduce
polynomial generalizations of these degree-two invariants.
More precisely, to each long virtual knot $K$, 
we associate two Laurent polynomials 
$V_1(K;t)$ and $V_2(K;t)$, 
defined in terms of Gauss diagrams and 
intersection numbers arising from surface realizations of $K$. 
This construction encodes both combinatorial and 
geometric information in a unified manner. 
Our first main result asserts that these polynomials
are well-defined invariants of long virtual knots (Theorem~\ref{thm31}).
Moreover, they recover the original degree-two invariants
by specialization at $t=1$:
\[
V_1(K;1)=v_{2,1}(K) \text{ and } V_2(K;1)=v_{2,2}(K).
\]

We establish symmetry properties, 
additivity under the concatenation product of long virtual knots, 
and prove that any pair of Laurent polynomials
can be realized as $(V_1(K;t),V_2(K;t))$ for some long virtual knot $K$ (Theorem~\ref{thm43}). 

We further investigate the finite type properties of the $V_1$- and $V_2$-polynomials. 
We show that $V_{1}(K;t)$ and $V_{2}(K;t)$ are not finite type invariants of any degree with respect to virtualizations,
while they remain finite type invariants of degree two
with respect to crossing changes (Theorem~\ref{thm51}). 
This highlights a sharp contrast
between the two notions of finite type theory
in the virtual setting.

We also study the behavior of the $V_1$- and $V_2$-polynomials under $\Delta$-moves. 
In classical knot theory, a well-known fact of Okada~\cite{Oka} states that a single $\Delta$-move changes 
the second coefficient of the Conway polynomial by $\pm1$. 
As a generalization of this fact, 
we show that a single $\Delta$-move
changes each of our polynomials by a monomial. 
As a consequence, 
the difference $V_1(K;t)-V_2(K;t)$
is invariant under $\Delta$-moves.
This yields a lower bound
for the $\Delta$-move distance between long virtual knots (Theorem~\ref{thm62}). 

Beyond the polynomials, we analyze their first derivatives at $t=1$.
We show that the integers $V_1'(K;1)$ and $V_2'(K;1)$
define finite type invariants of degree three with respect to virtualizations. 
In particular, we provide an explicit Gauss diagram formula for the invariant $V_1'(K;1)$. 

We further show that $V_1'(K;1)$ is invariant under welded moves. 
Consequently, it defines an invariant of long welded knots.
This enables us to relate $V_1'(K;1)$ to finite type invariants of long welded knots arising from the normalized Alexander polynomial.
The normalized Alexander polynomial $\Delta(K;t)$ of a long welded knot $K$, 
characterized by $\Delta(K;1)=1$ and $\Delta'(K;1)=0$,
gives rise to a sequence of invariants
\[
\alpha_n(K)=\left.\frac{1}{n!}\frac{d^n}{dt^n}\Delta(K;t)
\right|_{t=1}, 
\]
which are known to be finite type invariants of degree $n$
with respect to virtualizations~\cite{HKS,MY}.
Our main result in this direction shows that 
\[
V_1'(K;1)=\alpha_2(K)+\alpha_3(K) 
\] 
for any long welded knot $K$ (Theorem~\ref{thm74}). 
As a consequence, we derive an explicit Gauss diagram formula
for the $\alpha_3$-invariant (Corollary~\ref{cor76}). 

The paper is organized as follows.
In Section~\ref{sec2}, we recall basic definitions 
and known results on long virtual knots, intersection polynomials, and 
Goussarov-Polyak-Viro's invariants $v_{2,1}$ and $v_{2,2}$.
Section~\ref{sec3} introduces the polynomials 
$V_1(K;t)$ and $V_2(K;t)$, and proves that they
are well-defined invariants of long virtual knots. 
In Section~\ref{sec4}, we analyze their behavior 
under symmetries and product operations, and 
prove their realization theorem. 
Section~\ref{sec5} is devoted to finite type properties. 
In Section~\ref{sec6}, we study the behavior of 
the $V_1$- and $V_2$-polynomials under $\Delta$-moves.
Finally, in Section~\ref{sec7}, 
we prove that $V_1'(K;1)$ and $V_2'(K;1)$ are 
degree-three finite type invariants with respect to virtualizations, 
and establish the relationship between $V_1'(K;1)$ 
and the $\alpha_2$- and $\alpha_3$-invariants.


\section{Preliminaries}\label{sec2} 

In this section, we collect several basic notions and known results
that will be used throughout the paper.
We recall the definitions of long virtual knots, 
their equivalent descriptions,
the intersection polynomials $F_{ab}(K;t)$,
and the $v_{2,1}$- and $v_{2,2}$-invariants.

\subsection{Long virtual knots}
A \emph{long virtual knot diagram} 
is the image of a generic immersion of 
the real line $\mathbb{R}$ 
into the Euclidean plane $\mathbb{R}^{2}$ 
with finitely many transverse double points 
such that: 
\begin{enumerate}
\item
each transverse double point is designated as 
either a real crossing or virtual crossing, and 
\item
outside a compact set 
the image coincides with the $x$-axis of $\mathbb{R}^{2}$, that is, 
it agrees with the $x$-axis on $|x|>M$ for some sufficiently large~$M$. 
\end{enumerate}
See the top of Figure~\ref{3diagrams}. 

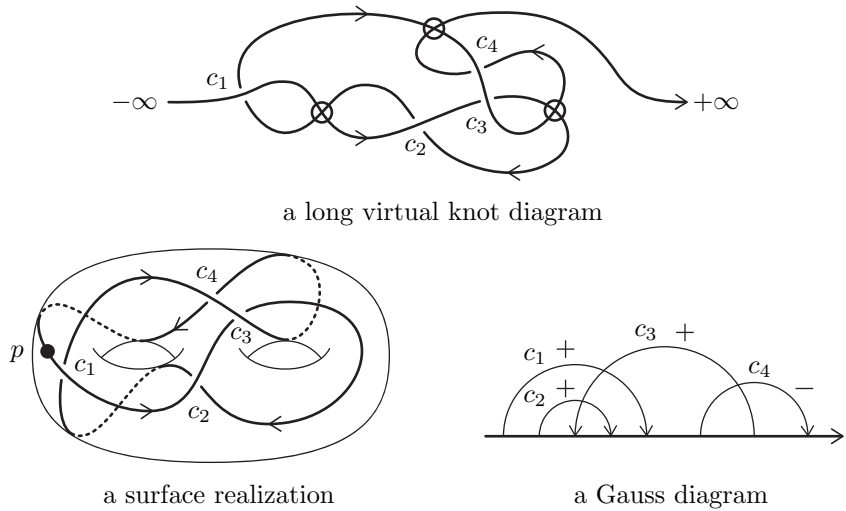
\begin{figure}[htbp]
  \centering
    \begin{overpic}[]{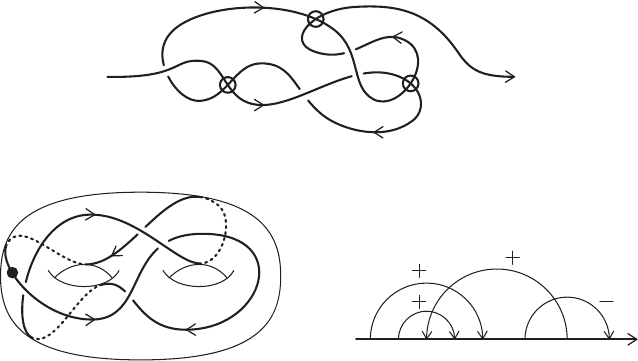}
      \put(30,134){$-\infty$}
      \put(250,134){$+\infty$}
      \put(66,143){$c_{1}$}
      \put(141,118){$c_{2}$}
      \put(164,126.5){$c_{3}$}
      \put(168,157){$c_{4}$}
      \put(95,92){a long virtual knot diagram}
      \put(-8,40){$p$}
      \put(16,35){$c_{1}$}
      \put(59,17){$c_{2}$}
      \put(75,47){$c_{3}$}
      \put(62,70){$c_{4}$}
      \put(27,-15){a surface realization}
      \put(186,39){$c_{1}$}
      \put(186,23){$c_{2}$}
      \put(228,48){$c_{3}$}
      \put(271,34){$c_{4}$}
      \put(205,-15){a Gauss diagram}
    \end{overpic}
\vspace{5mm}
  \caption{Three descriptions of a long virtual knot}
  \label{3diagrams}
\end{figure}

There are two equivalent descriptions of a long virtual knot diagram. 
One is a \emph{surface realization}, 
which is a knot diagram with a basepoint on a closed connected oriented surface. 
The other is a \emph{Gauss diagram}, which consists of a line 
equipped with finitely many oriented chords endowed with signs. 
These descriptions are illustrated at the bottom of Figure~\ref{3diagrams}. 
For further details, see \cite{CKS,GPV,Kau}. 

A \emph{long virtual knot} is an equivalence class of 
long virtual knot diagrams under Reidemeister moves 
I--VII. 
Equivalently, it can be regarded as an equivalence class of 
\begin{enumerate}
\item
surface realizations under Reidemeister moves I--III 
together with (de)stabilizations \cite{CKS}, or  
\item
Gauss diagrams under Reidemeister moves I--III~\cite{GPV}. 
\end{enumerate}
Throughout this paper, 
all long virtual knots are assumed to be oriented. 

\subsection{The intersection polynomials $F_{ab}(K;t)$}\label{subsec22}
For a long virtual knot $K$, 
twelve polynomial invariants 
\[
F_{ab}(K;t),\ G_{ab}(K;t),\ \text{and}\ H_{ab}(K;t) 
\quad (a,b\in\{0,1\}) 
\]
were introduced in~\cite{NNSW}. 
These invariants are collectively referred to as 
the \emph{intersection polynomials} of $K$. 
In this subsection, we recall the definition of $F_{ab}(K;t)$ 
using a surface realization of $K$. 

Let $(\Sigma,D)$ be a surface realization of $K$ 
with a basepoint~$p$, 
where $\Sigma$ is a closed connected oriented surface.  
Let $c_1,\dots,c_{n}$ be the crossings of $D$, 
and let $\e_i\in\{\pm1\}$ denote the sign of $c_i$ $(i=1,\dots,n)$. 
For each $i$, 
perform a smoothing operation at $c_i$, 
which yields a pair of closed curves on $\Sigma$. 
Let $\alpha_i$ denote the closed curve 
that does not contain the basepoint $p$,  
as shown in Figure~\ref{alpha}. 

\begin{figure}[htbp]
  \centering
    \begin{overpic}[]{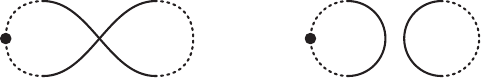}
      \put(-10,17){$p$}
      \put(44.5,28){$c_i$}
      \put(209,17){$\alpha_i$}
      \put(99.5,16){$\xrightarrow{\text{smoothing}}$}
    \end{overpic}
  \caption{The curve $\alpha_i$ on $\Sigma$}
  \label{alpha}
\end{figure}
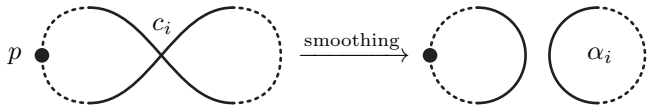

For $1\leq i,j\leq n$, 
the intersection number $\alpha_i\cdot\alpha_j\in\Z$ 
is defined by viewing $\alpha_i$ and $\alpha_j$ 
as homology classes in $H_1(\Sigma)$. 
The intersection number is 
invariant under (de)stabilizations of $\Sigma$. 

A crossing $c_i$ is said to be \emph{of type $0$} (resp. \emph{of type $1$}) 
if, when traveling along $D$ from the basepoint $p$ 
according to the orientation, 
the overcrossing (resp. undercrossing) of $c_i$ is encountered first.  
For $a\in\{0,1\}$, set 
\[I_a(D)=\{i\mid \mbox{$c_i$ is of type $a$}\}.\]
For $a,b\in\{0,1\}$, the Laurent polynomial is defined by 
\[F_{ab}(K;t)=
\sum_{i\in I_a(D),\, j\in I_b(D)}
\e_i\e_j(t^{\alpha_i\cdot\alpha_j}-1)\in{\Z}[t,t^{-1}].\]
It was shown in \cite[Theorem~2.2]{NNSW} that 
$F_{ab}(K;t)$ is an invariant of $K$. 
Moreover, $F_{ab}(K;t)$ is not a finite type invariant 
of any degree with respect to virtualizations, 
although it is a finite type invariant of degree two 
with respect to crossing changes 
\cite[Theorems~6.3 and 6.7]{NNSW}.

\subsection{The $v_{2,1}$- and $v_{2,2}$-invariants}\label{subsec23}

Let $D$ be a long virtual knot diagram 
representing a long virtual knot $K$.
When no confusion arises, we use the same symbol $D$
to denote the associated Gauss diagram. 
Likewise, we use the same notation $c_1,\dots,c_n$
for the crossings of $D$
and the corresponding chords of the Gauss diagram.

Two chords of $D$ are said to be \emph{linked} 
if their endpoints appear alternately 
along the underlying line of $D$. 
In particular, such a pair of chords is \emph{inwardly linked} 
(resp. \emph{outwardly linked}) 
if their orientations are as shown on the left (resp. right)  
of Figure~\ref{inward}. 
Set 
\begin{align*}
J_{\rm in}(D)&=\{(i,j)\mid 
\text{$c_i$ and $c_j$ are inwardly linked with 
$i\in I_0(D)$ and $j\in I_1(D)$}\}, \\
J_{\rm out}(D)&=\{(i,j)\mid 
\text{$c_i$ and $c_j$ are outwardly linked with 
$i\in I_1(D)$ and $j\in I_0(D)$}\}. 
\end{align*}

\begin{figure}[htbp]
  \centering
    \begin{overpic}[]{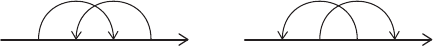}
      \put(-3,-12){inwardly linked chords}
      \put(111,-12){outwardly linked chords}
      \put(15,20){$c_{i}$}
      \put(67,20){$c_{j}$}
      \put(132,20){$c_{i}$}
      \put(186,20){$c_{j}$}
    \end{overpic}
\vspace{1.5em}
  \caption{Inwardly and outwardly linked chords}
  \label{inward}
\end{figure}
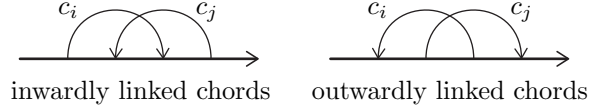

In~\cite[Section~3.2]{GPV}, Goussarov, Polyak, and Viro defined 
two integer-valued invariants of $K$ by 
\[v_{2,1}(K)=\sum_{(i,j)\in J_{\rm in}(D)}\e_i\e_j \text{ and }
v_{2,2}(K)=\sum_{(i,j)\in J_{\rm out}(D)}\e_i\e_j. 
\]
These are called the \emph{$v_{2,1}$-} and \emph{$v_{2,2}$-invariants} 
of $K$, respectively. 
They are finite type invariants of degree two 
with respect to virtualizations, 
and hence also with respect to crossing changes. 
Moreover, if $K$ is a long classical knot, 
then both $v_{2,1}(K)$ and $v_{2,2}(K)$ coincide with 
the second coefficient $a_{2}(K)$ of the Conway polynomial 
of the closure of $K$.


\section{The $V_1$- and $V_2$-polynomials}\label{sec3} 

The aim of this section is to introduce 
two polynomial invariants
by combining the intersection numbers
appearing in the definition of $F_{ab}(K;t)$
with the notions of inwardly and outwardly linked chords 
underlying the definitions of $v_{2,1}(K)$ and $v_{2,2}(K)$.

Let $K$ be a long virtual knot. 
For a diagram $D$ representing $K$,  
we define two Laurent polynomials by 
\[
V_1(D;t)=\sum_{(i,j)\in J_{\rm in}(D)}\e_i\e_j t^{\alpha_i\cdot\alpha_j} \text{ and }
V_2(D;t)=\sum_{(i,j)\in J_{\rm out}(D)}\e_i\e_j t^{\alpha_i\cdot\alpha_j}.
\]

\begin{theorem}\label{thm31}
The Laurent polynomials $V_1(D;t)$ and $V_2(D;t)$ 
are independent of the choice of any diagram $D$ representing $K$, 
and hence define invariants of $K$. 
\end{theorem}

\begin{proof}
It suffices to show that 
$V_1(D;t)=V_1(D';t)$ 
whenever two diagrams $D$ and $D'$ representing $K$ 
are related by a Reidemeister move I, II, or III. 
The proof for $V_2(D;t)$ is completely analogous. 

(I) Assume that $D'$ is obtained from $D$ 
by a Reidemeister move I, 
and let $c_1$ be the crossing of $D$ removed by this move. 
Since the two endpoints of the corresponding chord $c_1$ are adjacent 
on the underlying line, 
$c_1$ is not inwardly linked with any other chord. 
Therefore, $J_{\rm in}(D)=J_{\rm in}(D')$, 
and hence $V_1(D;t)=V_1(D';t)$.

(II) Assume that $D'$ is obtained from $D$ 
by a Reidemeister move II, 
removing the crossings $c_1$ and $c_2$ of $D$. 
These crossings have the same type. 
For any $i\in\{3,\dots,n\}$, 
the corresponding chord $c_1$ is inwardly linked with 
$c_i$ if and only if $c_2$ is. 
Since $\e_1=-\e_2$ and $\alpha_1=\alpha_2$, 
their contributions to $V_{1}(D;t)$ 
cancel pairwise. 
More precisely, we have 
\[\begin{cases}
\e_1\e_it^{\alpha_1\cdot\alpha_i}+
\e_2\e_i t^{\alpha_2\cdot\alpha_i}=0 
& \text{if $1,2\in I_0(D)$},\\
\e_i\e_1t^{\alpha_i\cdot\alpha_1}+
\e_i\e_2 t^{\alpha_i\cdot\alpha_2}=0 
& \text{if $1,2\in I_1(D)$}.
\end{cases}\]
Thus $V_1(D;t)=V_1(D';t)$. 

(III) Assume that $D'$ is obtained from $D$ 
by a Reidemeister move III, 
involving the crossings $c_1$, $c_2$, and $c_3$ of $D$. 
By case (II), it suffices to consider two cases 
shown in Figure~\ref{RIIIdiagram}, 
where the three arcs around the triangular region 
are oriented coherently. 
For each $i\in\{1,\dots,n\}$, 
the homology class $\alpha_i$ defined at $c_i$ in $D$ 
coincides with that defined at $c_i$ in $D'$. 

\begin{figure}[]
  \centering
    \begin{overpic}[]{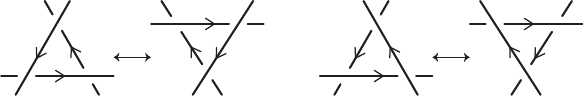}
      \put(58,22.5){III}
      \put(211.5,22.5){III}
    \end{overpic}
  \caption{Two Reidemeister moves III}
  \label{RIIIdiagram}
\end{figure}

There are six corresponding descriptions (i)--(vi) 
of these Reidemeister moves on Gauss diagrams, 
as shown in Figure~\ref{RIIIgauss}. 
In cases (i)--(iv), 
we have $J_{\rm in}(D)=J_{\rm in}(D')$, 
and hence $V_1(D;t)=V_1(D';t)$. 
In the remaining cases (v) and (vi), 
we have $\alpha_1=\alpha_2+\alpha_3$ 
and $\e_1=-\e_2=\e_3=\e$. 
Since 
$\alpha_1\cdot\alpha_3=(\alpha_2+\alpha_3)\cdot\alpha_3
=\alpha_2\cdot\alpha_3$, 
it follows that 
\[V_1(D;t)-V_1(D';t)=
\begin{cases}
\e_1\e_3t^{\alpha_1\cdot\alpha_3}+\e_2\e_3t^{\alpha_2\cdot\alpha_3}=0 
& \text{for case (v),} \\
\e_3\e_1t^{\alpha_3\cdot\alpha_1}+\e_3\e_2t^{\alpha_3\cdot\alpha_2}=0
& \text{for case (vi)}. 
\end{cases}\]
\end{proof}

\begin{figure}[htbp]
  \centering
  \vspace{1em}
    \begin{overpic}[]{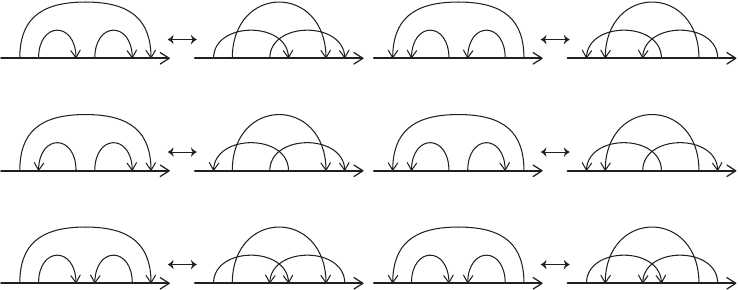}
      \put(82.75,126){(i)}
      \put(261,126){(ii)}
      \put(80,72.5){(iii)}
      \put(259.5,72.5){(iv)}
      \put(81.5,18){(v)}
      \put(259.5,18){(vi)}
      \put(39.5,142){$\e$}
      \put(21,127.5){$-\e$}
      \put(48,127.5){$-\e$}
      \put(133,142){$\e$}
      \put(100,125.5){$-\e$}
      \put(156,125.5){$-\e$}
      \put(219,142){$\e$}
      \put(201,127.5){$-\e$}
      \put(228,127.5){$-\e$}
      \put(312.5,142){$\e$}
      \put(279,125.5){$-\e$}
      \put(335,125.5){$-\e$}

      \put(39.5,87){$\e$}
      \put(25.5,73){$\e$}
      \put(48,73){$-\e$}
      \put(133,87){$\e$}
      \put(107,71){$\e$}
      \put(156,71){$-\e$}
      \put(219,87){$\e$}
      \put(201,73){$-\e$}
      \put(232.5,73){$\e$}
      \put(312.5,87){$\e$}
      \put(279,71){$-\e$}
      \put(339,71){$\e$}

      \put(39.5,33){$\e$}
      \put(21,19){$-\e$}
      \put(52.5,19){$\e$}
      \put(23.5,33.5){$c_{1}$}
      \put(23.5,8){$c_{2}$}
      \put(51,8){$c_{3}$}
      \put(133,33){$\e$}
      \put(100,17){$-\e$}
      \put(158,17){$\e$}
      \put(118,33.5){$c_{1}$}
      \put(117,8){$c_{2}$}
      \put(144,8){$c_{3}$}
      \put(219,33){$\e$}
      \put(201,19){$-\e$}
      \put(232.5,19){$\e$}
      \put(204,33.5){$c_{1}$}
      \put(231,8){$c_{2}$}
      \put(203.5,8){$c_{3}$}
      \put(312.5,33){$\e$}
      \put(279,17){$-\e$}
      \put(339,17){$\e$}
      \put(298,33.5){$c_{1}$}
      \put(324,8){$c_{2}$}
      \put(297,8){$c_{3}$}
    \end{overpic}
  \caption{Six descriptions on Gauss diagrams}
  \label{RIIIgauss}
\end{figure}
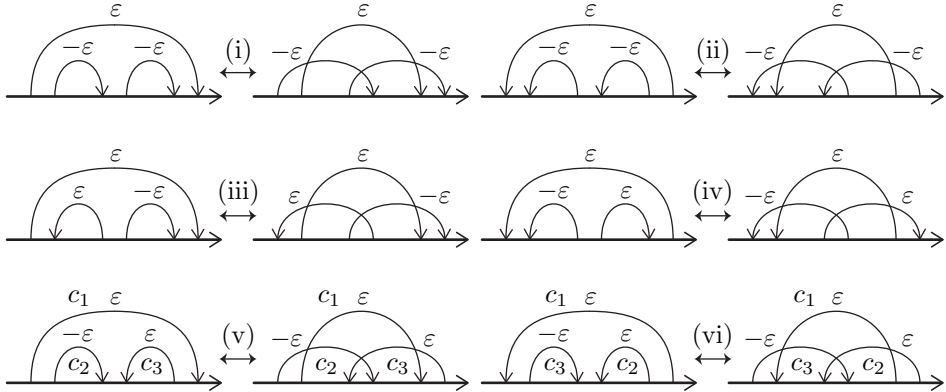

\begin{definition}
The \emph{$V_1$-polynomial} and the \emph{$V_2$-polynomial} 
of a long virtual knot $K$ are defined by 
\[
V_{1}(K;t)=V_{1}(D;t) \text{ and }
V_{2}(K;t)=V_{2}(D;t),
\]
where $D$ is any diagram representing $K$. 
\end{definition}

\begin{remark}
(i) By definition,  
\[V_1(K;1)=v_{2,1}(K)\text{ and }
V_2(K;1)=v_{2,2}(K).\]
Thus, the $V_1$- and $V_2$-polynomials 
can be regarded as polynomial generalizations of 
the $v_{2,1}$- and $v_{2,2}$-invariants, 
respectively. 

(ii) If $K$ is a long classical knot, 
then it admits a surface realization on $S^{2}$. 
Since any intersection number vanishes 
and $v_{2,1}(K)=v_{2,2}(K)=a_2(K)$, 
we have 
\[V_1(K;t)=V_2(K;t)=a_2(K),\]
where $a_2(K)$ is the second coefficient 
of the Conway polynomial of $K$. 
\end{remark}

For each chord $c_i$ of a Gauss diagram $D$ with sign $\e_i$, 
we assign signs to its endpoints so that 
the terminal endpoint has sign $\e_i$
and the initial endpoint has sign $-\e_i$. 
Let $\ell_i$ and $r_i$ denote the left and right endpoints of $c_i$, respectively, 
as shown in Figure~\ref{fig:endpoint-sign}. 

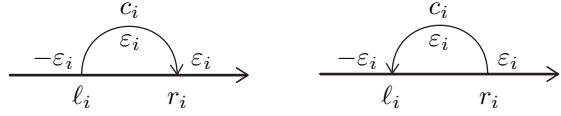
\begin{figure}[htbp]
  \centering
  \vspace{0.5em}
    \begin{overpic}[]{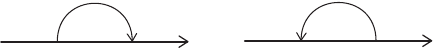}
      \put(42,26){$c_{i}$}
      \put(42,14){$\e_{i}$}
      \put(9,7){$-\e_{i}$}
      \put(69,7){$\e_{i}$}
      \put(24,-8){$\ell_{i}$}
      \put(60,-8){$r_{i}$}
      \put(159,26){$c_{i}$}
      \put(159,14){$\e_{i}$}
      \put(124,7){$-\e_{i}$}
      \put(185,7){$\e_{i}$}
      \put(141,-8){$\ell_{i}$}
      \put(178,-8){$r_{i}$}
    \end{overpic}
  \vspace{0.75em}
  \caption{The signs of the endpoints of a chord}
  \label{fig:endpoint-sign}
\end{figure}

Assume that $c_i$ and $c_j$ are 
inwardly or outwardly linked chords, 
and that their endpoints are arranged on the underlying line
in the order $\ell_i, \ell_j, r_i, r_j$. 
Let $c_{k}\neq c_{i},c_{j}$ be a chord 
whose endpoints satisfy one of the conditions (i)--(iii) below, 
as shown in Figure~\ref{fig:int-calculation}. 
\begin{enumerate}
\item
$\ell_k$ lies on the interval $\overline{\ell_i\ell_j}$ 
and $r_k$ lies on $\overline{\ell_jr_i}$. 
\item
$\ell_k$ lies on the interval $\overline{\ell_j r_i}$ 
and $r_k$ lies on $\overline{r_ir_j}$. 
\item
$\ell_k$ lies on the interval $\overline{\ell_i\ell_j}$ 
and $r_k$ lies on $\overline{r_ir_j}$.  
\end{enumerate}

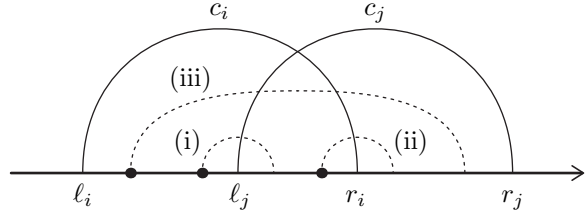
\begin{figure}[htbp]
  \centering
  \vspace{0.5em}
    \begin{overpic}[]{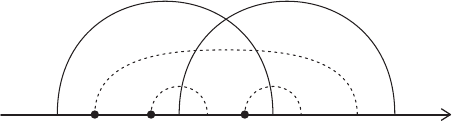}
      \put(75.5,63){$c_{i}$}
      \put(24,-8){$\ell_{i}$}
      \put(127,-8){$r_{i}$}
      \put(134,63){$c_{j}$}
      \put(83,-8){$\ell_{j}$}
      \put(186,-8){$r_{j}$}
      \put(62,12){(i)}
      \put(60.3,36){(iii)}
      \put(145,12){(ii)}
    \end{overpic}
  \vspace{0.75em}
  \caption{Three conditions (i), (ii), and (iii)}
  \label{fig:int-calculation}
\end{figure}

Let $S_{ij}$ denote the sum of signs 
assigned to $\ell_k$ for all chords $c_k$ 
satisfying one of the conditions (i), (ii), or (iii) above. 
Then the intersection number $\alpha_i\cdot\alpha_j$ 
is given by 
\[\alpha_i\cdot\alpha_j=
\begin{cases}
S_{ij} & \text{if the signs of $\ell_i$ and $\ell_j$ are opposite}, \\
S_{ij}+1 & \text{if the signs of $\ell_i$ and $\ell_j$ are both positive}, \\
S_{ij}-1 & \text{if the signs of $\ell_i$ and $\ell_j$ are both negative}.
\end{cases}\]
See \cite[Lemma 3.2]{HNNS1} for further details.

\begin{example}
Let $K$ be the long virtual knot 
represented by the Gauss diagram $D$ 
with four chords $c_{1},\dots,c_{4}$ 
shown in Figure~\ref{fig:example}. 
One readily verifies that 
\[J_{\rm in}(D)=\{(1,2), (1,4), (3,4)\}\text{ and }
J_{\rm out}(D)=\{(2,3)\}.\] 
Consider the pair $(1,4)\in J_{\rm in}(D)$. 
For the inwardly linked chords $c_1$ and $c_4$, 
the chords $c_2$ and $c_3$ satisfy the condition (iii), 
and the signs of their left endpoints $\ell_2$ and $\ell_3$ are both positive. 
Hence, we have $S_{14}=1+1=2$. 
Moreover, since the signs of $\ell_1$ and $\ell_4$ are both positive, 
it follows that $\alpha_1\cdot\alpha_4=2+1=3$. 
Using $\e_1=-1$ and $\e_4=+1$, 
we have $\e_1\e_4t^{\alpha_1\cdot\alpha_4}=-t^3$. 
A similar computation shows that 
$\e_1\e_2t^{\alpha_1\cdot\alpha_2}=\e_3\e_4t^{\alpha_3\cdot\alpha_4}=-t$. 
Therefore, $V_1(K;t)=-t^3-2t$. 
In the same manner, we obtain $V_2(K;t)=\e_2\e_3t^{\alpha_2\cdot\alpha_3}=-t$. 
\end{example}

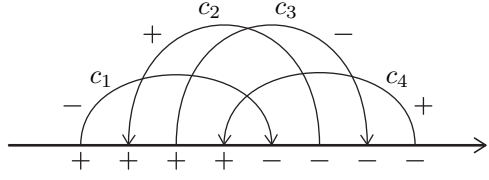
\begin{figure}[htb]
  \centering
  \vspace{0.75em}
    \begin{overpic}[]{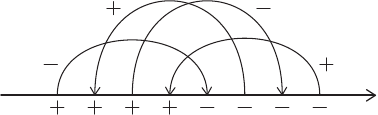}
      \put(31,33){$c_{1}$}
      \put(72,59){$c_{2}$}
      \put(101,59){$c_{3}$}
      \put(143,33){$c_{4}$}
    \end{overpic}
  \caption{A Gauss diagram with four chords $c_{1},\dots,c_{4}$}
  \label{fig:example}
\end{figure}


\section{Fundamental properties}\label{sec4} 

In this section, we establish fundamental properties of the
polynomials $V_1(K;t)$ and $V_2(K;t)$.
We first analyze their behavior under symmetries of long virtual knots 
and their additivity under the product operation.
We then show that these invariants realize 
any pair of Laurent polynomials.

Let $D=(\Sigma,D)$ be a surface realization 
of a long virtual knot $K$. 
Denote by $-D$, $D^*$, and $D^{\#}$ 
the diagrams obtained from $D$ 
by reversing the orientation of $D$, 
by reversing the orientation of $\Sigma$, 
and by switching the crossing information at each crossing, 
respectively. 
We further denote by 
$-K$, $K^*$, and $K^{\#}$ 
the long virtual knots represented by 
$-D$, $D^*$, and $D^{\#}$, respectively. 
When $D$ is regarded as a Gauss diagram, 
the diagram $-D$ is obtained from $D$ 
by reversing the orientation of the underlying line, 
$D^*$ is obtained by changing the signs of all chords, 
and $D^\#$ is obtained by changing both the signs and the orientations 
of all chords. 

\begin{lemma}\label{lem41}
For a long virtual knot $K$, 
the following hold. 
\begin{enumerate}
\item
$V_1(-K;t)=V_1(K;t^{-1})$ and $V_2(-K;t)=V_2(K;t^{-1})$. 
\item
$V_1(K^*;t)=V_1(K;t^{-1})$ and $V_2(K^*;t)=V_2(K;t^{-1})$. 
\item
$V_1(K^{\#};t)=V_2(K;t)$ and $V_2(K^{\#};t)=V_1(K;t)$. 
\end{enumerate}
\end{lemma}

\begin{proof}
The proof follows the same strategy 
as that of \cite[Theorem~5.2]{NNSW}. 

(i) For any $1\leq i,j\leq n$, we have $(i,j)\in J_{\rm in}(-D)$ 
if and only if $(j,i)\in J_{\rm in}(D)$. 
Moreover, the cycles at $c_i$ and $c_j$ in $-D$ 
are $-\alpha_i$ and $-\alpha_j$, respectively, 
while the signs of $c_i$ and $c_j$ remain unchanged. 
Therefore, 
\[V_1(-D;t)=\sum_{(i,j)\in J_{\rm in}(-D)}
\e_i\e_j t^{(-\alpha_i)\cdot(-\alpha_j)}
=\sum_{(j,i)\in J_{\rm in}(D)}
\e_j\e_i t^{-\alpha_j\cdot\alpha_i}
=V_1(D;t^{-1}).\]
The equality $V_2(-K;t)=V_2(K;t^{-1})$ is proved similarly. 

(ii) 
For any $1\leq i,j\leq n$, we have $(i,j)\in J_{\rm in}(D^*)$ 
if and only if $(i,j)\in J_{\rm in}(D)$. 
The intersection number between the cycles at $c_i$ and $c_j$ in $D^*$ 
is $-\alpha_i\cdot\alpha_j$, 
and the signs of $c_i$ and $c_j$ are $-\e_i$ and $-\e_j$, respectively. 
Hence, 
\[V_1(D^*;t)=\sum_{(i,j)\in J_{\rm in}(D^*)}
(-\e_i)(-\e_j) t^{-\alpha_i\cdot\alpha_j}
=\sum_{(i,j)\in J_{\rm in}(D)}
\e_i\e_j t^{-\alpha_i\cdot\alpha_j}
=V_1(D;t^{-1}).\]
The equality $V_2(K^*;t)=V_2(K;t^{-1})$ 
follows in the same manner. 

(iii) 
For any $1\leq i,j\leq n$, we have $(i,j)\in J_{\rm in}(D^{\#})$ 
if and only if $(i,j)\in J_{\rm out}(D)$. 
The cycles at $c_i$ and $c_j$ in $D^{\#}$ 
are $\alpha_i$ and $\alpha_j$, respectively, 
and the signs of $c_i$ and $c_j$ are $-\e_i$ and $-\e_j$. 
Thus, 
\[V_1(D^{\#};t)=\sum_{(i,j)\in J_{\rm in}(D^{\#})}
(-\e_i)(-\e_j) t^{\alpha_i\cdot\alpha_j}
=\sum_{(i,j)\in J_{\rm out}(D)}
\e_i\e_j t^{\alpha_i\cdot\alpha_j}
=V_2(D;t).\]
The equality $V_2(K^{\#};t)=V_1(K;t)$ follows similarly. 
\end{proof}

Let $D$ and $D'$ be diagrams representing 
long virtual knots $K$ and $K'$, respectively. 
Let $D\circ D'$ denote the diagram obtained by 
connecting $D$ and $D'$ in this order. 
The long virtual knot represented by $D\circ D'$ 
is called the \emph{product} of $K$ and $K'$, 
and is denoted by $K\circ K'$. 

\begin{lemma}\label{lem42}
For $s=1,2$, we have
$V_s(K\circ K';t)=V_s(K;t)+V_s(K';t)$. 
\end{lemma}

\begin{proof}
By the definition of the product, 
\[J_{\rm in}(D\circ D')=J_{\rm in}(D)\sqcup J_{\rm in}(D') \text{ and }
J_{\rm out}(D\circ D')=J_{\rm out}(D)\sqcup J_{\rm out}(D').\]
Moreover, 
for any $(i,j)\in J_{\rm in}(D\circ D')\sqcup J_{\rm out}(D\circ D')$, 
the intersection number of the cycles at $c_i$ and $c_j$ 
in $D\circ D'$ coincides with that in $D$ or $D'$. 
The claim follows immediately. 
\end{proof}

The following theorem implies that 
the $V_1$- and $V_2$-polynomials are 
independent of each other. 

\begin{theorem}\label{thm43}
For any Laurent polynomials $f(t),g(t)\in{\Z}[t,t^{-1}]$, 
there exists a long virtual knot $K$ such that 
$V_1(K;t)=f(t)$ and $V_2(K;t)=g(t)$. 
\end{theorem}

\begin{proof}
For each integer $n\geq 0$, 
consider the long virtual knots $K(n)$ and $K'(n)$ 
represented by the Gauss diagrams $D(n)$ and $D'(n)$, respectively, 
shown in Figure~\ref{fig:Dn}. 
Here, $c_{1},\dots,c_{n+2}$ (resp. $c'_{1},\dots,c'_{n+2}$) 
the chords of $D(n)$ (resp. $D'(n)$). 
The only difference between $D(n)$ and $D'(n)$ 
is that the chords $c_{2}$ and $c'_{2}$ have opposite signs. 

\begin{figure}[htbp]
  \centering
  \vspace{0.75em}
    \begin{overpic}[]{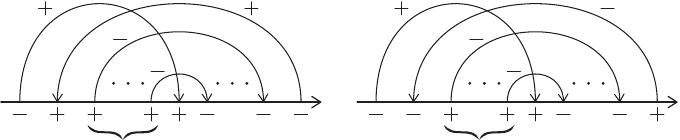}
      \put(-11,50){$D(n)$}
      \put(43,70){$c_{1}$}
      \put(88,70){$c_{2}$}
      \put(88,56.5){$c_{3}$}
      \put(88,36){$c_{n+2}$}
      \put(56.5,-7){$n$}
      \put(316,50){$D'(n)$}
      \put(213.5,71){$c'_{1}$}
      \put(260,71){$c'_{2}$}
      \put(260,56){$c'_{3}$}
      \put(260,37){$c'_{n+2}$}
      \put(228.5,-7){$n$}
    \end{overpic}
  \vspace{0.75em}
  \caption{The Gauss diagrams $D(n)$ and $D'(n)$}
  \label{fig:Dn}
\end{figure}
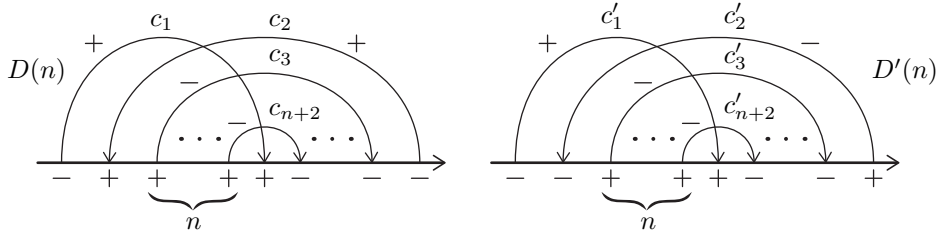

One readily verifies that 
\[J_{\rm in}(D(n))=J_{\rm in}(D'(n))=\{(1,2)\}
\text{ and }J_{\rm out}(D(n))=J_{\rm out}(D'(n))=\emptyset.\]
Since $\alpha_1\cdot\alpha_2=n$ and $\alpha_1'\cdot\alpha_2'=n-1$, 
we obtain 
\[V_1(K(n);t)=t^n, \ V_1(K'(n);t)=-t^{n-1}, \text{ and }
V_2(K(n);t)=V_2(K'(n);t)=0.\] 
By Lemmas~\ref{lem41}(i) and \ref{lem42}, 
there exists a long virtual knot $K$ 
such that \[V_1(K;t)=f(t)\text{ and }V_2(K;t)=0.\]
Moreover, by Lemma~\ref{lem41}(iii), 
there exists a long virtual knot $K'$ 
such that \[V_1(K';t)=0\text{ and }V_2(K';t)=g(t).\]
Therefore, the product $K\circ K'$ is a long virtual knot 
satisfying that 
\[
V_1(K\circ K';t)=f(t) \text{ and } V_2(K\circ K';t)=g(t).
\] 
\end{proof}


\section{Finite type invariants}\label{sec5} 

In this section, we investigate the finite type properties 
of the $V_1$- and $V_2$-polynomials. 
We consider finite type invariants with respect to 
two local operations on long virtual knots: 
virtualizations and crossing changes.

\begin{theorem}\label{thm51}
The following statements hold. 
\begin{enumerate}
\item
The $V_{1}$- and $V_{2}$-polynomials 
are not finite type invariants of any degree 
with respect to virtualizations. 
\item
The $V_{1}$- and $V_{2}$-polynomials 
are finite type invariants of degree two 
with respect to crossing changes. 
\end{enumerate}
\end{theorem}

Let $D$ be a Gauss diagram representing a long virtual knot $K$, 
and let $C=\{c_1,\dots,c_k\}$ be a subset of the chords of $D$. 
For $\delta_1,\dots,\delta_k\in\{0,1\}$, 
denote by $D_{\delta_1,\dots,\delta_k}$ 
the Gauss diagram obtained from $D$ by 
removing all the chords $c_{i}$ with $\delta_{i}=1$, 
and by $K_{\delta_1,\dots,\delta_k}$ 
the long virtual knot represented by $D_{\delta_1,\dots,\delta_k}$. 

To show that the $V_1$-polynomial is not 
a finite type invariant of any degree 
with respect to virtualizations, 
it suffices to show the following lemma. 
We refer to \cite{GPV} for the precise definition of finite type invariants 
with respect to virtualizations.

\begin{lemma}\label{lem52}
Let $K(n)$ $(n\geq 0)$ be the long virtual knot 
given in the proof of {\rm Theorem~\ref{thm43}}. 
For the subset $C=\{c_3,\dots,c_{n+2}\}$ of chords of $D(n)$, 
we have 
\[\sum_{\delta_3,\dots,\delta_{n+2}\in\{0,1\}}
(-1)^{\delta_3+\dots+\delta_{n+2}}
V_1(K(n)_{\delta_3,\dots,\delta_{n+2}};t)\ne 0.\]
\end{lemma}

\begin{proof}
Let $s$ denote the number of zeros among 
$\delta_3,\dots,\delta_{n+2}\in\{0,1\}$. 
Since the diagram $D(n)_{\delta_3,\dots,\delta_{n+2}}$ 
represents the long virtual knot $K(s)$, 
we obtain 
\begin{align*}
&\sum_{\delta_3,\dots,\delta_{n+2}\in\{0,1\}}
(-1)^{\delta_3+\dots+\delta_{n+2}}
V_1(K(n)_{\delta_3,\dots,\delta_{n+2}};t)\\
&=\sum_{s=0}^n(-1)^{n-s}\binom{n}{s}
V_1(K(s);t)=(t-1)^n\ne 0. 
\end{align*}
\end{proof}

We next consider finite type invariants with respect to crossing changes.  
In this case, we again use the notation $D_{\delta_1,\dots,\delta_k}$ 
to denote the Gauss diagram obtained from $D$ 
by changing both the signs and the orientations of 
all chords $c_{i}$ with $\delta_{i}=1$. 

\begin{lemma}\label{lem53}
For any Gauss diagram $D$ representing 
a long virtual knot $K$ 
and any crossings $c_1,c_2,c_3$ of $D$, 
we have
\[\sum_{\delta_1,\delta_2,\delta_3\in\{0,1\}}
(-1)^{\delta_1+\delta_2+\delta_3}
V_1(K_{\delta_1,\delta_2,\delta_3};t)=0.\]
\end{lemma}

\begin{proof}
The proof follows the same strategy 
as that of \cite[Lemma~8.3]{NNSW}. 
For $1\leq i,j\leq n$ and $\delta_1,\delta_2,\delta_3\in\{0,1\}$, 
define a Laurent polynomial by 
\[
\Psi_{i,j}^{\delta_1,\delta_2,\delta_3}(t)=
\begin{cases}
\e_i\e_jt^{\alpha_i\cdot\alpha_j} 
& \text{for $(i,j)\in J_{\rm in}(D_{\delta_1,\delta_2,\delta_3})$}, \\
0 & \text{otherwise}. 
\end{cases}\]
Then the alternating sum can be written as 
\begin{align*}
&\sum_{\delta_1,\delta_2,\delta_3\in\{0,1\}}
(-1)^{\delta_1+\delta_2+\delta_3}
\left(\sum_{(i,j)\in J_{\rm in}(D_{\delta_1,\delta_2,\delta_3})}
\e_i\e_jt^{\alpha_i\cdot\alpha_j}\right)\\
&=\sum_{1\leq i,j\leq n}
\left(\sum_{\delta_1,\delta_2,\delta_3\in\{0,1\}}
(-1)^{\delta_1+\delta_2+\delta_3}
\Psi_{i,j}^{\delta_1,\delta_2,\delta_3}(t)
\right).
\end{align*}

Thus it suffices to show that, 
for each fixed pair $(i,j)$, 
the inner sum vanishes. 
By the definition of $\Psi_{i,j}^{\delta_1,\delta_2,\delta_3}(t)$, 
we observe the following. 
\begin{itemize}
\item
If $4\leq i,j\leq n$, 
then $\Psi_{i,j}^{\delta_1,\delta_2,\delta_3}(t)$ 
is independent of $\delta_1,\delta_2,\delta_3$. 
\item
If $1\leq i\leq 3$ and $4\leq j\leq n$, then 
$\Psi_{i,j}^{\delta_1,\delta_2,\delta_3}(t)$ is independent of 
$\delta_k$, where $k\in\{1,2,3\}\setminus\{i\}$. 
\item
If $(i,j)$ with $4\leq i\leq n$ and $1\leq j\leq 3$, then
$\Psi_{i,j}^{\delta_1,\delta_2,\delta_3}(t)$ is independent of 
$\delta_k$, where $k\in\{1,2,3\}\setminus\{j\}$. 
\item
If $(i,j)$ with $1\leq i,j\leq 3$, then
$\Psi_{i,j}^{\delta_1,\delta_2,\delta_3}(t)$ is independent of 
$\delta_k$, where $k\in\{1,2,3\}\setminus\{i,j\}$. 
\end{itemize}
In each case, the inner sum is zero, 
and hence the claim follows. 
\end{proof}

\begin{proof}[Proof of {\rm Theorem~\ref{thm51}.}] 
(i) By Lemma~\ref{lem52}, 
the $V_1$-polynomial is not a finite type invariant 
of any degree with respect to virtualizations. 
Lemma~\ref{lem41}(iii) implies that  
the same conclusion holds for the $V_2$-polynomial. 

(ii) 
By Lemma~\ref{lem53}, 
the $V_1$-polynomial is 
a finite type invariant of degree at most two 
with respect to crossing changes. 
To show that the degree is exactly two, 
consider the long virtual knot $K$ 
represented by the Gauss diagram $D$ 
consisting of two positive chords $c_{1}$ and $c_{2}$, 
as shown in Figure~\ref{fig:pf-degree2}. 
Set $C=\{c_{1},c_{2}\}$. 
A direct computation yields that 
\[
V_1(K_{0,0};t)-V_1(K_{0,1};t)-V_1(K_{1,0};t)+V_1(K_{1,1};t)
=1-0-0+0=1\ne 0.
\]
Hence, the $V_1$-polynomial is a finite type invariant of degree exactly two with respect to crossing changes. 
By Lemma~\ref{lem41}(iii), 
the same statement holds for the $V_2$-polynomial. 
\end{proof}

\begin{figure}[htbp]
  \centering
  \vspace{0.5em}
    \begin{overpic}[]{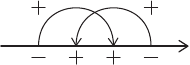}
      \put(30,32){$c_{1}$}
      \put(53,32){$c_{2}$}
    \end{overpic}
  \caption{A Gauss diagram $D$ with two chords $c_{1}$ and $c_{2}$}
  \label{fig:pf-degree2}
\end{figure}


\section{Delta-moves}\label{sec6} 

It is known that, 
if two long classical knots $K$ and $K'$ differ by 
a $\Delta$-move, then $|a_2(K)-a_2(K')|=1$; 
see~\cite{Oka}. 
As a generalization of this fact, 
the $V_1$- and $V_2$-polynomials 
satisfy the following property. 

\begin{proposition}\label{prop61}
If two long virtual knots $K$ and $K'$ 
are related by a single $\Delta$-move, 
then there exist $\e\in\{\pm 1\}$ and 
$n\in{\Z}$ such that 
\[V_1(K;t)-V_1(K';t)=V_2(K;t)-V_2(K';t)=\e t^n.\]
In particular, the difference $V_1(K;t)-V_2(K;t)$ is 
invariant under $\Delta$-moves. 
\end{proposition}

\begin{proof}
By Lemma~\ref{lem41}(i), (ii), together with the invariance 
under Reidemeister moves II, 
it suffices to consider the $\Delta$-move shown in Figure~\ref{fig:Delta}, 
where $c_1,c_2,c_3$ denote the three chords 
involved in this move. 
A direct computation shows that 
\[V_1(K;t)-V_1(K';t)=-t^{\alpha_1\cdot\alpha_3}\text{ and }
V_2(K;t)-V_2(K';t)=-t^{\alpha_3\cdot\alpha_2}.\] 
Since $\alpha_1+\alpha_2=\alpha_3$, 
we have $\alpha_1\cdot\alpha_3=\alpha_3\cdot\alpha_2$,  
and hence, 
\[V_1(K;t)-V_1(K';t)=V_2(K;t)-V_2(K';t).\]
\end{proof}

\begin{figure}[htbp]
  \centering
    \begin{overpic}[]{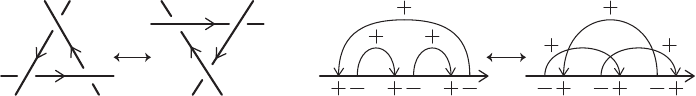}
      \put(60,22){$\Delta$}
      \put(23,-13){$K$}
      \put(95,-13){$K'$}
      \put(240,22){$\Delta$}
      \put(177,14.5){$c_{1}$}
      \put(204.5,14.5){$c_{2}$}
      \put(177,40){$c_{3}$}
      \put(191,-13){$K$}
      \put(277.5,14.5){$c_{1}$}
      \put(304,14.5){$c_{2}$}
      \put(277.5,40){$c_{3}$}
      \put(290,-13){$K'$}
    \end{overpic}
  \vspace{1em}
  \caption{A $\Delta$-move}
  \label{fig:Delta}
\end{figure}
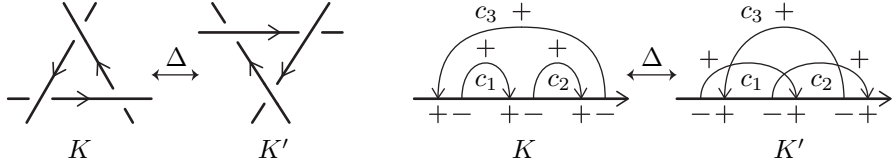

In general, two long virtual knots $K$ and $K'$ 
are not necessarily related by a finite sequence of $\Delta$-moves. 
For example, by Proposition~\ref{prop61}, 
the knots $K(n)$ and $K'(n)$ 
introduced in the proof of Theorem~\ref{thm43} 
cannot be related by $\Delta$-moves; 
in fact, we have 
\[V_1(K(n);t)-V_2(K(n);t)\ne V_1(K'(n);t)-V_2(K'(n);t).\]
When $K$ and $K'$ are related by $\Delta$-moves, 
let $d_{\Delta}(K,K')$ denote 
the minimal number of $\Delta$-moves 
required to transform $K$ into $K'$. 

\begin{theorem}\label{thm62}
Let $K$ and $K'$ be long virtual knots that are 
related by a finite sequence of $\Delta$-moves. 
If $V_1(K;t)-V_1(K';t)=\sum_{i \in \mathbb{Z}}a_i t^i$, 
then 
\[
d_{\Delta}(K,K') \geq \sum_{i \in \mathbb{Z}} |a_i|.
\]
\end{theorem}

\begin{proof}
Set $d = d_{\Delta}(K,K')$. 
By definition, there exists a sequence of long virtual knots 
\[K = K_0, K_1, K_2, \dots, K_d = K'\]
such that each $K_{j}$ is obtained from $K_{j-1}$ 
by a single $\Delta$-move for $j = 1, \dots, d$. 
By Proposition~\ref{prop61}, 
we have $||V_1(K_{j-1};t)-V_1(K_j;t)||=1$, 
where 
we define $||f(t)||=\sum_{i\in{\Z}}|c_i|$ 
for a Laurent polynomial $f(t)=\sum_{i\in{\Z}}c_it^i$. 
Therefore, 
\[\sum_{i\in{\Z}}|a_i|=||V_1(K;t)-V_1(K';t)||
\leq \sum_{j=1}^d||V_1(K_{j-1};t)-V_1(K_j;t)||=d.\]
\end{proof}


\section{The invariants $V_1'(K;1)$ 
and $V_2'(K;1)$}\label{sec7} 

Although the $V_{1}$- and $V_{2}$-polynomials are not 
finite type invariants of any degree with respect to 
virtualizations, 
their values at $t=1$, 
namely $V_1(K;1)=v_{2,1}(K)$ and $V_2(K;1)=v_{2,2}(K)$, 
are finite type invariants of degree two~\cite{GPV}. 
It is therefore natural to ask whether higher-order
finite type information can be extracted from these polynomials. 
In this section, 
we investigate the integers $V_1'(K;1)$ and $V_2'(K;1)$, 
the first derivatives of 
$V_{1}(K;t)$ and $V_{2}(K;t)$ evaluated at $t=1$.

\begin{lemma}\label{lem71}
The invariant $V_1'(K;1)$ 
admits the following Gauss diagram formula: 
\begin{align*}
V_1'(K;1)&=
\langle D_1,D\rangle-\langle D_2,D\rangle
+\langle D_3,D\rangle-\langle D_4,D\rangle\\
&\quad+\langle D_5,D\rangle-\langle D_6,D\rangle
+\langle D_7,D\rangle-\langle D_8,D\rangle. 
\end{align*}
Here, $D_1,\dots,D_8$ 
are the Gauss diagrams shown in {\rm Figure~\ref{fig:pairing}}, 
and the pairing $\langle D_s,D\rangle$ 
counts the number of subdiagrams of $D$ matching $D_s$, 
weighted by the product of the signs of 
the chords appearing in $D_s$ $(s=1,\dots,8)$. 
\end{lemma}

\begin{figure}[htbp]
  \centering
    \begin{overpic}[]{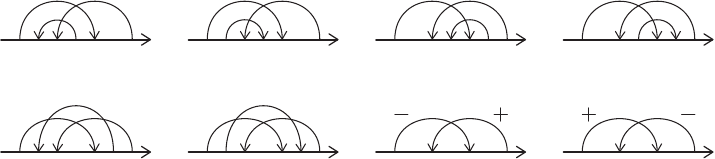}
      \put(32,45){$D_1$}
      \put(122,45){$D_2$}
      \put(212,45){$D_3$}
      \put(302,45){$D_4$}
      \put(32,-11){$D_5$}
      \put(122,-11){$D_6$}
      \put(212,-11){$D_7$}
      \put(302,-11){$D_8$}
    \end{overpic}
  \vspace{1em}
  \caption{The Gauss diagrams $D_1,\dots,D_8$}
  \label{fig:pairing}
\end{figure}
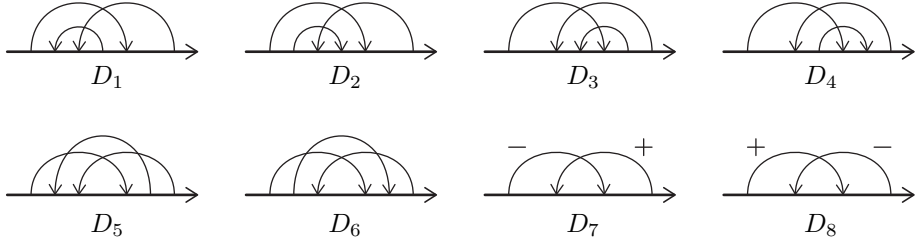

\begin{proof}
By definition, we have $V_1'(K;1)=\sum_{(i,j)\in J_{\rm in}(D)}
\e_i\e_j(\alpha_i\cdot\alpha_j)$. 
The claim therefore follows from expressing  
the intersection number $\alpha_i\cdot\alpha_j$  
in terms of Gauss diagrams, 
as explained in Section~\ref{sec3}.  
More precisely, 
we have 
\[\sum_{(i,j)\in J_{\rm in}(D)}\e_i\e_jS_{ij}=
\langle D_1,D\rangle-\langle D_2,D\rangle
+\langle D_3,D\rangle-\langle D_4,D\rangle
+\langle D_5,D\rangle-\langle D_6,D\rangle.\]
Moreover, 
if the signs of $\ell_i$ and $\ell_j$ are both positive (resp. both negative), 
then $\e_i=-1$ and $\e_j=+1$ (resp. 
$\e_i=+1$ and $\e_j=-1$). 
Therefore, 
\[\sum_{\substack{(i,j)\in J_{\rm in}(D)\\ \ell_i,\ell_j:\,\text{positive}}} \e_i\e_j=
\langle D_7,D\rangle \mbox{ and }
\sum_{\substack{(i,j)\in J_{\rm in}(D)\\ \ell_i,\ell_j:\,\text{negative}}} (-\e_i\e_j)=
-\langle D_8,D\rangle.\]
\end{proof}

We remark that a similar Gauss diagram formula 
for $V_2'(K;1)$ can be obtained from 
Lemmas~\ref{lem41}(iii) and \ref{lem71}; 
the details are left to the reader.

\begin{lemma}\label{lem72}
The invariants $V_1'(K;1)$ and $V_2'(K;1)$ are  
finite type invariants of degree three 
with respect to virtualizations. 
\end{lemma}

\begin{proof}
By Lemma~\ref{lem41}(iii), 
it suffices to show the statement  for $V_1'(K;1)$. 
The argument is similar to that of Lemma~\ref{lem53}. 
Consider the pairing 
$\langle D_1,D\rangle$. 
For $1\leq i,j,k\leq n$ and 
$\delta_1,\delta_2,\delta_3,\delta_4\in\{0,1\}$, 
define an integer by 
\[
\Psi_{i,j,k}^{\delta_1,\delta_2,\delta_3,\delta_4}=
\begin{cases}
\e_i\e_j\e_k
& \text{if $c_i,c_j,c_k$ are the three chords of 
$D_1\subset D_{\delta_1,\delta_2,\delta_3,\delta_4}$}, \\
0 & \text{otherwise}. 
\end{cases}\]
Then 
\begin{align*}
&\sum_{\delta_1,\delta_2,\delta_3,\delta_4\in\{0,1\}}
(-1)^{\delta_1+\delta_2+\delta_3+\delta_4}
\langle D_1,D_{\delta_1,\delta_2,\delta_3,\delta_4}\rangle\\
&=\sum_{1\leq i,j,k\leq n}
\left(\sum_{\delta_1,\delta_2,\delta_3,\delta_4\in\{0,1\}}
(-1)^{\delta_1+\delta_2+\delta_3+\delta_4}
\Psi_{i,j,k}^{\delta_1,\delta_2,\delta_3,\delta_4}
\right).
\end{align*}
Since $\Psi_{i,j,k}^{\delta_1,\delta_2,\delta_3,\delta_4}$ 
is independent of some $\delta_{\ell}$ 
with $\ell\in\{1,2,3,4\}\setminus\{i,j,k\}$, 
the inner alternating sum vanishes. 
The same argument applies to the other pairings 
$\langle D_s,D\rangle$. 
Hence, we obtain 
\[\sum_{\delta_1,\delta_2,\delta_3,\delta_4\in\{0,1\}}
(-1)^{\delta_1+\delta_2+\delta_3+\delta_4}
V_1'(K_{\delta_1,\delta_2,\delta_3,\delta_4};1)=0.\]

To see that the degree is exactly three, 
consider the long virtual knot $K$ 
represented by the Gauss diagram $D_1$ in Figure~\ref{fig:pairing}, 
where all three chords, say $c_1,c_2,c_3$, 
have positive signs. 
Set $C=\{c_1,c_2,c_3\}$. 
A direct computation shows that 
\[\sum_{\delta_1,\delta_2,\delta_3\in\{0,1\}}
(-1)^{\delta_1+\delta_2+\delta_3}
V_1'(K_{\delta_1,\delta_2,\delta_3};1)=
V_1'(K_{0,0,0};1)=V_1'(K;1)=1.\]
Thus, $V_1'(K;1)$ is a finite type invariant 
of degree exactly three 
with respect to virtualizations. 
\end{proof}

A \emph{long welded knot} is an equivalence class 
of long virtual knots under welded moves. 
At the level of Gauss diagrams, 
a \emph{welded move} exchanges 
two consecutive initial endpoints of chords. 
The welded theory was first introduced 
by Fenn, Rim\'{a}nyi, and Rourke~\cite{FRR} 
in the context of braids. 
By definition, the invariant $V_1(K;1)=v_{2,1}(K)$ 
extends to an invariant of long welded knots 
and is a finite type invariant of degree two; 
see~\cite{HKS,MY}. 

\begin{lemma}\label{lem73}
The invariant $V_1'(K;1)$ is invariant 
under welded moves, 
and hence defines an invariant of long welded knots. 
\end{lemma}

\begin{proof}
Let $D$ and $D'$ be Gauss diagrams representing 
long virtual knots $K$ and $K'$, respectively, 
which differ by a single welded move. 
For each $s\in\{1,4,7,8\}$, the welded move induces a natural bijection 
between subdiagrams of $D$ and of $D'$ matching $D_{s}$. 
Consequently, $\langle D_s,D\rangle=\langle D_s,D'\rangle$. 
On the other hand, we obtain 
\begin{align*}
&\langle D_2,D\rangle+\langle D_6,D\rangle=
\langle D_2,D'\rangle+\langle D_6,D'\rangle, \\
&\langle D_3,D\rangle+\langle D_5,D\rangle=
\langle D_3,D'\rangle+\langle D_5,D'\rangle. 
\end{align*}
Therefore, by Lemma~\ref{lem71}, we conclude that 
$V_1'(K;1)=V_1'(K';1)$. 
\end{proof}

Let $\Delta(K;t)$ denote 
the normalized Alexander polynomial 
of a long welded knot~$K$, 
characterized by $\Delta(K;1)=1$ and $\Delta'(K;1)=0$. 
For $n\geq 2$, a sequence of invariants is defined by 
\[\alpha_n(K)=
\left.\frac{1}{n!}\frac{d^n}{dt^n}\Delta(K;t)\right|_{t=1},\] 
called the \emph{$\alpha_n$-invariant} of $K$. 
In particular, $\alpha_2(K)=v_{2,1}(K)$. 
The invariant $\alpha_n(K)$ is a finite type invariant 
of degree $n$ with respect to virtualizations. 
We refer the reader to~\cite{HKS,MY} for further details.

\begin{theorem}\label{thm74}
For any long welded knot $K$, 
we have 
\[V_1'(K;1)=\alpha_2(K)+\alpha_3(K).\]
\end{theorem}

\begin{proof}
The space of finite type invariants 
of degree three for long welded knots 
with respect to virtualizations 
is spanned by $\{1,\alpha_2,\alpha_3\}$; 
see~\cite{BD,MY}. 
By Lemmas~\ref{lem72} and \ref{lem73}, 
there exist $p,q,r\in{\Q}$ such that 
\[V_1'(K;1)=p+q\alpha_2(K)+r\alpha_3(K).\]  

Let $O$ denote the trivial long welded knot, 
and let $K_7$ be the long welded knot 
represented by the Gauss diagram $D_7$ in Figure~\ref{fig:pairing}. 
The values of $V_1'(K;1)$, $\alpha_2(K)$, and $\alpha_3(K)$ 
for $K=O,K_7,-K_7$ are listed in Table~\ref{table:values}. 
From the relations
\[0=p+0+0, \ -1=p-q+0, \text{ and }
1=p-q+2r,\]
we obtain $p=0$ and $q=r=1$. 
\end{proof}

\begin{center}
\begin{table}[htb]
\caption{Values of the invariants for $K=O,K_7,-K_7$}
\begin{tabular}{c|cc|ccc}
$K$ & $V_1(K;t)$ & $V_1'(K;1)$ & $\Delta(K;t)$ 
& $\alpha_2(K)$ & $\alpha_3(K)$ \\
\hline \hline
$O$ & $0$ & $0$ & $1$ & $0$ & $0$ \\
\hline
$K_7$ & $-t$ & $-1$ & $-t^2+2t$ & $-1$ & $0$ \\
\hline
$-K_7$ & $-t^{-1}$ & $1$ & 
$2t^{-1}-t^{-2}$ & $-1$ & $2$ \\
\end{tabular}
\label{table:values}
\end{table}
\end{center}

We remark that any long welded knot $K$ satisfies that 
\[V_1'(-K;1)=-V_1'(K;1), \ 
\alpha_2(-K)=\alpha_2(K), \text{ and }
\alpha_3(-K)=-2\alpha_2(K)-\alpha_3(K).
\]

\begin{example}\label{ex75}
Consider the long welded knot $K(n)$ 
represented by the Gauss diagram $D(n)$ in Figure~\ref{fig:Dn}. 
Since $V_1(K(n);t)=t^n$, 
we have 
\[V_1'(K(n);1)=n\text{ and }
\alpha_2(K(n))=V_1(K(n);1)=1.\]
By Theorem~\ref{thm74}, we obtain 
$\alpha_3(K(n))=n-1$. 

Indeed, the normalized Alexander polynomial 
of $K(n)$ is given by 
\[\Delta(K(n);t)=t^{n+1}-2t^n+t^{n-1}+1.\]
Since  $\Delta'''(K(n);1)=6(n-1)$, 
we obtain $\alpha_3(K(n))=n-1$. 
\end{example}

\begin{corollary}\label{cor76}
The $\alpha_{3}$-invariant admits the following Gauss diagram formula: 
\begin{align*}
\alpha_{3}(K)&=
\langle D_1,D\rangle-\langle D_2,D\rangle
+\langle D_3,D\rangle-\langle D_4,D\rangle \\
&\quad+\langle D_5,D\rangle-\langle D_6,D\rangle
-2\langle D_8,D\rangle-\langle D_9,D\rangle
-\langle D_{10},D\rangle, 
\end{align*}
where $D_{9}$ and $D_{10}$ are the Gauss diagrams shown in 
{\rm Figure~\ref{fig:pairing2}}. 
\end{corollary}

\begin{figure}[htbp]
  \centering
    \begin{overpic}[]{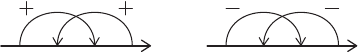}
      \put(32,-11){$D_9$}
      \put(129,-11){$D_{10}$}
    \end{overpic}
  \vspace{1em}
  \caption{The Gauss diagrams $D_9$ and $D_{10}$}
  \label{fig:pairing2}
\end{figure}

\begin{proof}
Since the Gauss diagram formula of the $\alpha_{2}$-invariant is given by 
\[
\alpha_{2}(K)=
\langle D_7,D\rangle+\langle D_8,D\rangle
+\langle D_9,D\rangle+\langle D_{10},D\rangle, 
\]
the claim follows from 
Theorem~\ref{thm74}. 
\end{proof}


\end{document}